\documentclass[12pt]{amsart}
\usepackage{amscd}
\usepackage{amssymb}

\def\NZQ{\Bbb}               

\def\ZZ{{\NZQ Z}}

\def\frk{\frak}               

\def\pp{{\frk p}}

\def\qq{{\frk q}}

\def\mm{{\frk m}}

\def\Phi{{\frk n}}
\def\Phi{{\frk N}}
\def\opn#1#2{\def#1{\operatorname{#2}}} 
\opn\chara{char} \opn\length{\ell} \opn\pd{pd} \opn\rk{\lk}\opn\link{link}
\opn\projdim{proj\,dim} \opn\injdim{inj\,dim} \opn\rank{rank}\opn\Var{Var}
\opn\depth{depth} \opn\and{and} \opn\grade{grade}
\opn\height{height} \opn\embdim{emb\,dim} \opn\codim{codal}

\opn\Tr{Tr} \opn\bigrank{bigrank}
\opn\superheight{superheight}\opn\lcm{lcm}

\opn\trdeg{trdeg}
\opn\reg{reg} \opn\lreg{lreg} \opn\ini{in} \opn\mod{mod}
\opn\div{div} \opn\Div{Div} \opn\cl{cl} \opn\Cl{Cl}
\opn\Spec{Spec} \opn\Supp{Supp} \opn\supp{supp} \opn\Sing{Sing}
\opn\Ass{Ass} \opn\Min{Min} \opn\Var{Var}
\opn\Ann{Ann} \opn\Rad{Rad} \opn\Soc{Soc} \opn\H{H}
\opn\Im{Im} \opn\Spec{Spec
}
\opn\inf{inf}
 \opn\Ker{Ker} \opn\Coker{Coker} \opn\Am{Am} \opn \inf{inf}
\opn\Hom{Hom} \opn\Tor{Tor} \opn\Ext{Ext} \opn\End{End} \opn\cd{cd}
\opn\Aut{Aut} \opn\id{id}

\opn\nat{nat}
\opn\pff{pf}
\opn\Pf{Pf} \opn\GL{GL} \opn\SL{SL} \opn\mod{mod} \opn\ord{ord}
\opn\cl{cl} \opn\conv{conv} \opn\ext{ext} \opn\rad{rad}
\opn\star{star} \opn\red{red}\opn\H{H} \opn\bight{bight}
\opn\aff{aff} \opn\con{conv} \opn\relint{relint} \opn\st{st}
\opn\lk{lk} \opn\cn{cn} \opn\core{core} \opn\vol{vol}
\opn\link{link} \opn\star{star}  \opn\mdepth{mdepth} \opn\mgrade{mgrade}
\opn\gr{gr}

\def\pot#1#2{#1[\kern-0.28ex[#2]\kern-0.28ex]}

\opn\dirlim{\underrightarrow{\lim}}
\opn\inivlim{\underleftarrow{\lim}}
\let\tensor=\otimes
\let\iso=\cong

\let\Dirsum=\bigoplus

\let\to=\rightarrow

\def\Implies{\ifmmode\Longrightarrow \else
     \unskip${}\Longrightarrow{}$\ignorespaces\fi}
\def\implies{\ifmmode\Rightarrow \else
     \unskip${}\Rightarrow{}$\ignorespaces\fi}
\def\iff{\ifmmode\Longleftrightarrow \else
     \unskip${}\Longleftrightarrow{}$\ignorespaces\fi}

\let\:=\colon
\newtheorem{Theorem}{Theorem}[section]
\newtheorem{Lemma}[Theorem]{Lemma}
\newtheorem{Corollary}[Theorem]{Corollary}
\newtheorem{Proposition}[Theorem]{Proposition}
\newtheorem{Remark}[Theorem]{Remark}

\newtheorem{Example}[Theorem]{Example}

\newtheorem{Definition}[Theorem]{Definition}

\newtheorem{Question}[Theorem]{Question}
\newtheorem{Fact}[Theorem]{Fact}

\let\epsilon\varepsilon
\let\phi=\varphi
\let\kappa=\varkappa
\def\qed{\ifhmode\textqed\fi
   \ifmmode\ifinner\quad\qedsymbol\else\dispqed\fi\fi}
\def\textqed{\unskip\nobreak\penalty50
    \hskip2em\hbox{}\nobreak\hfil\qedsymbol
    \parfillskip=0pt \finalhyphendemerits=0}
\def\dispqed{\rlap{\qquad\qedsymbol}}

\opn\dis{dis}
\def\pnt{{\raise0.5mm\hbox{\large\bf.}}}

\begin{document}

\title{ maximal depth property of bigraded modules}

\author{ Ahad Rahimi}

.

\address{ Ahad Rahimi, Department of Mathematics, Razi University, Kermanshah, Iran}\email{ahad.rahimi@razi.ac.ir}

\begin{abstract}
 Let $S=K[x_1, \dots, x_m, y_1, \dots, y_n]$ be the standard
bigraded polynomial ring over a field $K$. Let $M$ be a finitely generated bigraded $S$-module and $Q=(y_1, \dots,  y_n)$.  We say $M$ has maximal depth with respect to $Q$ if there is an associated prime $\pp$ of $M$  such that $\grade(Q, M)=\cd(Q, S/\pp)$. In this paper,  we study finitely generated bigraded modules with maximal depth with respect to $Q$. It is shown that sequentially Cohen--Macaulay modules with respect to $Q$ have maximal depth with respect to $Q$. In fact, maximal depth property generalizes the concept of sequentially Cohen--Macaulayness. Next, we show that if $M$ has maximal depth with respect to $Q$ with $\grade(Q, M)>0$, then $H^{\grade(Q, M)}_{Q}(M)$ is not finitely generated. As a consequence, "generalized Cohen--Macaulay modules with respect to $Q$" having "maximal depth with respect to $Q$" are Cohen--Macaulay with respect to $Q$.  All hypersurface rings that have maximal depth with respect to $Q$ are classified.
\end{abstract}

\maketitle

\section*{Introduction}

Let $K$ be a field and $S=K[x_1, \dots, x_m, y_1, \dots, y_n]$ be the standard
bigraded
 polynomial ring over $K$. In other words, $\deg x_i=(1,0)$ and $\deg y_j=(0, 1)$  for all $i$ and $j$.
 We set the bigraded irrelevant ideals $P=(x_1, \dots, x_m)$ and $Q=(y_1, \dots,  y_n)$. Let $M$ be a
finitely
 generated bigraded $S$-module. The author has been studying the algebraic properties of a finitely generated bigraded $S$-module $M$  with respect to $Q$, see for instance \cite{JR},  \cite{R1},  \cite{R2}.
We denote by $\cd(Q, M)$ the {\em cohomological dimension of $M$ with
respect to $Q$ } which is the largest integer $i$ for which $H^i_ Q
(M)\neq 0$.

A classical fact in commutative algebra (\cite{JR}) says that
\[
\grade(Q, M) \leq \min\{\cd(Q, S/\pp): \pp \in \Ass(M)\}.
\]
 We set $\mgrade(Q, M)=\min\{\cd(Q, S/\pp): \pp \in \Ass(M)\}$. We say $M$ has {\em maximal depth with respect to $Q$} if the equality holds, i.e., $\grade(Q, M) =\mgrade(Q, M).$ In other words, there is an associated prime $\pp$ of $M$  such that $\grade(Q, M) =\cd(Q, R/\pp)$. Some examples
 of modules with maximal depth with respect to $Q$ are given in Example \ref{examples}.  In this paper, the author studies depth property for  finitely generated bigraded modules.

We let $P=0$ and consider $R=K[y_1, \dots, y_n]$ as standard graded polynomial ring. Then, $M$ as ordinary graded $R$-module has maximal depth if $\depth M= \mdepth M$ where $\mdepth M=\min \{ \dim R/\pp: \pp \in \Ass(M) \}$.  This concept has already been working with several authors. Some known results in this regard are as follows: If $I \subseteq R$ is a generic monomial ideal, then it has maximal depth, see \cite[Theorem 2.2]{MSY}. If a monomial ideal $I$ has maximal depth, then so does its polarization, see \cite{FT}. Edge ideals of cycle graphs, transversal polymatroidal ideals and high powers of connected bipartite graph with maximal depth property are classified in \cite{R4}.

The paper is organized as follows: In the preliminary section, we give some facts about $\mgrade(Q, M)$. Let $\mathcal{F}$: $0=M_0\varsubsetneq M_1 \varsubsetneq \dots  \varsubsetneq M_d=M$ be the {\em dimension filtration} of $M$ with respect to $Q$. We observe that all the $M_i$ have the same $\mgrade$ with respect to $Q$, namely this number is $\cd(Q, M_1)$. From this fact, we deduce that if $M$ is {\em sequentially Cohen--Macaulay with respect to $Q$},  then $M$ has maximal depth with respect to $Q$, see Proposition~\ref{mdepth}.  Of course this class is rather large.  In Example \ref{notseq},  the ring $R$ is not sequentially Cohen--Macaulay with respect to $Q$ and has maximal depth with respect to $Q$. If $P=0$, we deduce that the ordinary sequentially Cohen--Macaulay modules have maximal depth, see \cite{R3}.

Let $(R, \mm)$ be a Noetherian local ring and $M$ a finitely generated $R$-module. This is a known fact that if there exists $\pp \in \Ass(M)$ such that $\dim R/\pp=j>0$, then $H^j_{\mm}(M)$
is not finitely generated. Inspired by this fact, we may ask the following question:
\begin{Question}
\label{notfg}
Assume that there exists $\pp \in \Ass(M)$ such that $\cd(Q, S/\pp)=j>0$. Does it follow that $H^j_Q(M)$ is not finitely generated?
\end{Question}

In Section 2, we give a positive answer to this question in a particular case. We first show that if $M$ has maximal depth with respect to $Q$ with $\grade(Q, M)>0$ and $|K|=\infty$, then there exists a
bihomogeneous $M$-regular element $y\in Q$ of degree $(0,1)$ such
that $M/yM$ has maximal depth with respect to $Q$. An example is given to show that the converse is not true in general. This fact is used to answer the above question in the following case:  if $M$ has maximal depth with respect to $Q$ with $\grade(Q, M)>0$, then $H^{\grade(Q, M)}_{Q}(M)$ is not finitely generated, see Theorem~\ref{gradenotfg}.

As a consequence, "generalized Cohen--Macaulay modules with respect to $Q$" with "maximal depth with respect to $Q$" are Cohen--Macaulay with respect to $Q$. In fact, we show: If  $M$ is generalized Cohen--Macaulay with respect to $Q$  with $\grade(Q, M)>0$, then $M$ has maximal depth with respect to $Q$ is equivalent to say that "$M$ is sequentially Cohen--Macaulay with respect to $Q$" and this is the same as "$M$ is Cohen--Macaulay with respect to $Q$". If $P=0$, we deduce that ordinary generalized Cohen--Macaulay modules with maximal depth are Cohen--Macaulay, see \cite{R3}.

In the following section, we let $I\subseteq S$ be a monomial ideal. It is shown that $\mgrade(Q, S/I)=n-d$ where $d$ is the maximal height of an associated prime of $I$ in $Q$. Moreover, if $S/I$ is Cohen--Macaulay, then $S/I$ has maximal depth with respect to $P$ and $Q$. We also show that the maximal depth property is preserved under tensor product and direct sum.

In the final section,  we classify all hypersurface rings that have maximal depth with respect to $Q$.
\section{Preliminaries}

 Let $K$ be a field and $S=K[x_1, \dots, x_m, y_1, \dots, y_n]$ be the standard
bigraded
 polynomial ring over $K$. In other words, $\deg x_i=(1,0)$ and $\deg y_j=(0, 1)$  for all $i$ and $j$.
 We set the bigraded irrelevant ideals
$P=(x_1, \dots, x_m)$ and $Q=(y_1, \dots,  y_n)$. Let $M$ be a
finitely
 generated bigraded $S$-module.
We denote by $\cd(Q, M)$ the {\em cohomological dimension of $M$ with
respect to $Q$ } which is the largest integer $i$ for which $H^i_ Q
(M)\neq 0$. Let $|K|=\infty$. In \cite[Proposition 1.7]{JR} it is shown that
\begin{equation}
 \label{mgrade}
\grade(Q, M) \leq \min\{\cd(Q, S/\pp): \pp \in \Ass(M)\}.
\end{equation}
 We set $\mgrade_S(Q, M)=\min\{\cd(Q, S/\pp): \pp \in \Ass(M)\}$. For simplicity, we write $\mgrade(Q, M)$ instead of $\mgrade_S(Q, M)$.
We recall the following facts which will be used in the sequel.
 \begin{Fact}	
\label{cd}{\em
		The following statements hold.
\begin{itemize}
\item[{(a)}]$\cd(Q,M)  =  \max \{\cd (Q, S/{\pp}): \pp \in \Ass(M)\}$, see \cite[Corollary 4.6]{CJR}.
\item[{(b)}]  $\cd(P, M)=\dim M/QM$  and $\cd(Q, M)=\dim M/PM$,  see \cite[Formula 3]{R1}.
\item[{(c)}] $\grade(Q, M)\leq \dim M-\cd(P, M)$, and the equality holds if $M$ is Cohen--Macaulay, see \cite[Formula 5]{R1}.
\item[{(d)}] $\grade(Q, M)=0$ if and only if there exists $\pp \in \Ass(M)$ such that $Q\subseteq \pp$.
\end{itemize}}
\end{Fact}
 Observe that
\[
\grade(Q,  M)\leq \mgrade(Q, M)\leq \cd(Q,M) \leq \dim M.
 \]
Fact \ref{cd}(a) provides the second inequality. Note that $\grade(Q, M)=0$ if and only if $\mgrade(Q, M)=0$. Thus, if $\mgrade(Q, M)=1$, then $\grade(Q, M)=1$.
\begin{Definition}{\em
We say $M$ has {\em maximal depth with respect to $Q$} if the equality (\ref{mgrade}) holds, i.e.,
\[
\grade(Q, M) =\mgrade(Q, M).
\]
In other words, there is an associated prime $\pp$ of $M$  such that $\grade(Q, M)=\cd(Q, S/\pp)$. }
\end{Definition}

\begin{Example}	
\label{examples}{\em Some examples of modules with maximal depth property are as follows:
\begin{itemize}
\item Let $q\in \ZZ$. In \cite{R1}, we say $M$ is Cohen--Macaulay with respect to $Q$ if we have only one non-vanishing local cohomology. In other words, $\grade(Q, M) = \cd(Q, M) = q$.  Cohen--Macaulay modules with respect to $Q$ have maximal depth with respect to $Q$ because $\grade(Q, M)=\cd(Q, S/\pp)$ for every associated prime $\pp$ of $M$.
\item If $\cd(Q, M)\leq 1$, then $M$ has maximal depth with respect to $Q$.
\item If $\grade(Q, M)=0$, then $M$ has maximal depth with respect to $Q$. In fact, Fact \ref{cd}(d) provides an associated prime $\pp$ of $M$ such that $Q\subseteq \pp$. Hence   $\cd(Q, S/\pp)=\dim S/(P+\pp)=\dim S/(P+Q)=0$. The first equality follows from Fact \ref{cd}(b). Therefore, $M$ has maximal depth with respect to $Q$.
\end{itemize} }
\end{Example}

A finite filtration $\mathcal{D}$: $0=D_0\varsubsetneq D_1
\varsubsetneq
 \dots  \varsubsetneq  D_r=M$ of bigraded submodules of $M$ is the
dimension filtration of $M$ with respect to $Q$ if $D_{i-1}$ is the
largest bigraded submodule of $D_i$ for which $\cd(Q,D_{i-1})<\cd(Q, D_i)$
for all $i=1, \dots, r$.
\begin{Fact}
\label{PR}{\em
let  $\mathcal{D}$ be the dimension filtration of $M$ with respect to $Q$.  Then
\begin{itemize}
\item[{(a)}]$
\Ass(D_i)=\{ \pp\in \Ass(M): \cd(Q, S/\pp)\leq \cd(Q, D_i)\}$, see \cite[Lemma 1.7]{PR}.
\item[{(b)}]$
\Ass(M/D_i)=\Ass(M)- \Ass(D_{i})$, see \cite[Corollary 1.10]{PR}.
\end{itemize} }
\end{Fact}
\begin{Lemma}
\label{assd}
Let  $\mathcal{D}$ be the dimension filtration of $M$ with respect to $Q$. Then
\[
\Ass(D_i/D_{i-1})=\{ \pp \in \Ass(M): \cd(Q, S/\pp)=\cd(Q, D_i)\}.
\]
In particular,
\begin{equation}
 \label{assu}	
\Ass(M) =\bigcup_{i=1}^r\Ass(D_i/D_{i-1}).
\end{equation}
\end{Lemma}
\begin{proof}
We set $A=\{ \pp \in \Ass(M): \cd(Q, S/\pp)=\cd(Q, D_i)\}$. Let $\pp \in \Ass(D_i/D_{i-1})$. By Fact \ref{cd}(a) we have $\cd(Q, S/\pp)\leq \cd(Q, D_i/D_{i-1})=\cd(Q, D_i)$. The embedding $0 \to D_i/D_{i-1} \to M/D_{i-1}$ also yields $\pp \in \Ass(M/D_{i-1})$. Hence $\pp \in \Ass(M)$ and $\pp \not \in \Ass(D_{i-1})$ by Fact \ref{PR}(b).
Consequently, $\cd(Q, D_{i-1})<\cd(Q, S/\pp)\leq \cd(Q, D_i)$. Fact \ref{PR}(a) provides the first inequality. It follows that $\cd(Q, S/\pp)=\cd(Q, D_i)$ and hence $\pp \in A$. Now let $\pp \in A$. Thus $\cd(Q, S/\pp)=\cd(Q, D_i)$ and so $\pp \in \Ass(D_i)$ by Fact \ref{PR}(a). As $\pp \not \in \Ass(D_{i-1})$, the containment $\Ass(D_i)\subseteq \Ass(D_{i-1})\cup \Ass(D_i/D_{i-1})$ implies $\pp \in \Ass(D_i/D_{i-1})$.
\end{proof}
A finite filtration $\mathcal{F}$:
$0=M_0\varsubsetneq M_1 \varsubsetneq
 \dots  \varsubsetneq M_r=M$
 of $M$ by
bigraded submodules $M$ is called a {\em Cohen--Macaulay filtration with
respect to $Q$} if each quotient $M_i/M_{i-1}$ is Cohen--Macaulay with respect to $Q$ and
$0 \leq \cd(Q, M_1/M_0)<\cd(Q, M_2/M_1)< \dots< \cd(Q, M_r/M_{r-1})$. If $M$ admits a Cohen--Macaulay filtration with respect to $Q$, then we say $M$ is {\em sequentially Cohen--Macaulay with respect to $Q$}. Note that if $M$ is sequentially Cohen--Macaulay with respect to $Q$, then the filtration $\mathcal{F}$ is uniquely determined and it is just the dimension filtration of $M$ with respect to $Q$, that is, $\mathcal{F}=\mathcal{D}$, see \cite{R2}.

\begin{Proposition}
\label{mdepth}
Let  $\mathcal{F}$: $0=M_0\varsubsetneq M_1 \varsubsetneq \dots \varsubsetneq M_d=M$ be the dimension filtration of $M$ with respect to $Q$.
 Then, $\mgrade(Q, M_i)=\cd(Q, M_1)$ for $i=1, \dots, d$.  Moreover, if $M$ is sequentially Cohen--Macaulay with respect to $Q$, then $M$ has maximal depth with respect to $Q$.
\end{Proposition}
\begin{proof}
We first show that $\mgrade(Q, M)=\cd(Q, M_1)$. We set $\mgrade(Q, M)=t$. Thus there exists $\pp \in \Ass(M)$ such that $\cd(Q, S/\pp)=t$. Hence  $\pp \in \Ass(M_i/M_{i-1})$ for some $i$ by (\ref{assu}). Thus $\cd(Q, S/\pp)=\cd(Q, M_i)=t$  again by (\ref{assu}). Note that $t=\mgrade(Q, M)\leq \cd(Q, M_1)$. If $t<\cd(Q, M_1)$, then $\cd(Q, M_i)<\cd(Q, M_1)$ for some $i$, a contradiction. Therefore, $\mgrade(Q, M)=\cd(Q, M_1)$.
Now we observe that
\[
t=\mgrade(Q, M)\leq \mgrade(Q, M_{d-1})\leq \dots \leq \mgrade(Q,  M_1)\leq \cd(Q, M_1)=t.
\]
Consequently,  $\mgrade(Q,  M_i)=t$ for $i=1, \dots, d$.

To show the second part, let $M$ be sequentially Cohen--Macaulay with respect to $Q$. Thus the dimension filtration of  $\mathcal{F}$ with respect to $Q$ is the Cohen--Macaulay filtration with respect to $Q$. As $M_1$ is Cohen--Macaulay with respect to $Q$, it has maximal depth with respect to $Q$ and so $\grade(Q, M_1)=\mgrade(Q, M_1)$. Since $M$ is sequentially Cohen--Macaulay with respect to $Q$, it follows that $\grade(Q,  M_i)=\grade(Q, M)$ for all $i$, see \cite[Fact 2.3]{R2}. Using the first part, we have
\[
\grade(Q, M)=\grade(Q, M_1)=\mgrade(Q, M_1)=\mgrade(Q, M),
\]
 as desired.
\end{proof}

In the following, we give an example which is not sequentially Cohen--Macaulay with respect to $Q$ and has maximal depth with respect to $Q$.
\begin{Example}
\label{notseq}{\em
Let $S=K[x_1, x_2, y_1, y_2, y_3, y_4]$ be the standard bigraded polynomial ring. We set $R=S/I$ where $I=( x_1x_2, x_1y_3, x_1 y_4, x_2y_1, y_1y_3, y_1y_4, y_2y_4, y_2y_3)$ and $Q=(y_1, y_2, y_3, y_4)$. By using CoCoA(\cite{Co}), the ideal $I$ has the minimal primary decomposition $I=\bigcap_{i=1}^{3}\pp_i$ where
$\pp_1=(x_1, y_1, y_2), \pp_2=(x_2, y_3, y_4)$ and $\pp_3=(x_1, y_1, y_3,y_4)$. Fact \ref{cd}(a) provides $\mgrade(Q, R)=1$. Thus $\grade(Q, R)=1$ and so $R$ has maximal depth with respect to $Q$. On the other hand, $R$ is not sequentially  Cohen--Macaulay with respect to $Q$, see \cite[Example 2.15]{NR1}.
}
\end{Example}

\section{Not finitely generated local cohomology modules}

Let $(R, \mm)$ be a Noetherian local ring and $M$ a finitely generated $R$-module. This is a known fact that if there exists $\pp \in \Ass(M)$ such that $\dim R/\pp=j>0$, then $H^j_{\mm}(M)$
is not finitely generated, see \cite[Corollary 11.3.3]{BS} and \cite[Exercise 11.3.9]{BS}. Inspired by this fact, we may ask the following question:
\begin{Question}
\label{notfg}
Assume that there exists $\pp \in \Ass(M)$ such that $\cd(Q, S/\pp)=j>0$. Does it follow that $H^j_Q(M)$ is not finitely generated?
\end{Question}

In this section, we have a positive answer for this question in a particular case. First, we prove the following crucial lemma:
\begin{Lemma}
\label{regular} Suppose $\grade(Q, M)>0$  and  $|K|=\infty$. If $M$ has maximal depth with respect to $Q$, then there exists a
bihomogeneous $M$-regular element $y\in Q$ of degree $(0,1)$ such
that $M/yM$ has maximal depth with respect to $Q$.
\end{Lemma}
\begin{proof}
Here we follow the proof of \cite[Proposition 1.7]{JR}. By our assumption, there exists $\pp \in \Ass(M)$ such that $\grade(Q, M)=\cd(Q, S/\pp)$. As  $\grade(Q,
M)>0$, there exists a bihomogeneous
$M$-regular element $y\in Q$  such that $\grade(Q, M/yM)=\grade(Q, M)-1$.  The element $\pp \in \Ass(M)$ is properly contain in an element $\qq \in \Ass (M/yM)$. The  element $y$ may be chosen to avoid
all the minimal prime ideal of $\Supp(S/(P+\pp))$, too.
Observe that
\begin{eqnarray*}
\grade(Q, M)-1 & = & \grade(Q, M/yM) \\
          & \leq & \cd(Q, S/\qq) \\
          &=& \dim S/(P+\qq)\\
          &<&  \dim S/(P+\pp)\\
          &=& \cd(Q, S/\pp)\\
          & =&  \grade(Q, M).
\end{eqnarray*}
Consequently,
 \begin{equation}
 \label{homogen}
\grade(Q, M/yM) = \cd(Q, S/\qq) \quad\text{where} \quad  \qq \in \Ass(M/yM).
\end{equation}
Therefore, $M/yM$ has maximal depth with respect to $Q$.
\end{proof}
The following example shows that the converse of Lemma \ref{regular} is not true in general.
\begin{Example}
\label{regularexample}{\em
Let $S=K[x_1,x_2,y_1, y_2]$ be the standard bigraded polynomial ring.
   We set  $R=S/(f)$ where $f=x_1y_1+x_2y_2$ and $Q=(y_1, y_2)$. The ring $R$ is Cohen--Macaulay of dimension $3$ and  $\grade(Q, R)=\dim R-\cd(P, R)=3-2=1>0$ by Fact \ref{cd}(c). The bihomogenous element $y=y_1+y_2\in S$ is $R$-regular. Consider the following isomorphism
 \[
 R/yR\iso S/(f, y).
 \]
 Using Macaulay2 (\cite{GSE}) gives us the associated primes of $T=S/(f, y)$ that is
 \[
 \{ (y_1, y_2), (y_1+y_2, x_1-x_2)\}.
 \]
Since $T$ has an associated prime contains $Q$, it follows from Fact \ref{cd}(d) that
\[
\grade(Q, T)=\grade(Q, R/yR)=0.
 \]
 Therefore, $R/yR$ has maximal depth with respect to $Q$.  On the other hand,  as $\Ass(R)=\{(f)\}$ we have
 \[
 \mgrade(Q, R)=\cd(Q, R)=\dim R/PR=\dim S/P=2.
 \]
Fact \ref{cd}(b) explains the second equality. As $\grade(Q, R)=1$,  the ring $R$ has no maximal depth with respect to $Q$. }
\end{Example}
\begin{Theorem}
\label{gradenotfg}
Assume $M$ has maximal depth with respect to $Q$ with $\grade(Q, M)>0$ and $|K|=\infty$. Then $H^{\grade(Q, M)}_{Q}(M)$ is not finitely generated.
\end{Theorem}
\begin{proof}
We set $\grade(Q, M)=r$. We proceed by induction on $r$.   Our assumption says that there exists $\pp\in \Ass(M)$ such that $\grade(Q, M)=\cd(Q, S/\pp)$. Suppose $r=1$. The exact sequence $0 \to S/\pp \to M \to U\to 0$ yields the exact sequence $0\to H^0_{Q}(U) \to  H^1_{Q}(S/\pp) \to H^1_{Q}(M) \to  H^1_{Q}(U) \to 0$. This is a known fact that for every finitely generated $S$-module $N$, the local cohomology module $H^{\cd(Q, N)}_Q(N)$ is not finitely generated. Now suppose $H^1_{Q}(M)$ is finitely generated. Then the exact sequence $0\to H^0_{Q}(U) \to  H^1_{Q}(S/\pp) \to K \to 0$ where $K$ is a submodule of $H^1_{Q}(M)$ provides $H^1_{Q}(S/\pp)$ is finitely generated, a contradiction. Therefore, $H^1_{Q}(M)$ is not finitely generated. Now suppose $r\geq 2$ and that the
statement holds for all modules $N$ which has maximal depth with respect to $Q$  and $\grade(Q, N)<r$.  We want to
prove it for $M$ which has maximal depth with respect
to $Q$ with $\grade(Q, M)=r$. By Lemma \ref{regular}, there exists a
bihomogeneous $M$-regular element $y\in Q$ of degree $(0,1)$ such
that $M/yM$ has maximal depth with respect to $Q$. In fact, by (\ref{homogen}) we have,
\[
\grade(Q, M/yM) = \cd(Q, S/\qq) \quad\text{for some} \quad  \qq \in \Ass(M/yM).
\]
The exact sequence $0 \to M \stackrel y\to M \to M/yM \to 0$ yields the exact sequence $0\to H^{r-1}_{Q}(M/yM) \to  H^r_{Q}(M) \stackrel y\to H^r_{Q}(M) $. Our induction hypothesis provides $H^{r-1}_{Q}(M/yM)$ is not finitely generated. Therefore, $H^r_{Q}(M)$ is not finitely generated too, as desired.
\end{proof}

Let  $M$ a finitely generated bigraded $S$-module. We say $M$ is {\em generalized Cohen--Macaulay with respect to $Q$} if the local cohomology module $H^i_{Q}(M)$ is finitely generated for all $i<\cd(Q, M)$. As a consequence, we have the following
\begin{Corollary}
Suppose $M$ is generalized Cohen--Macaulay with respect to $Q$  with $\grade(Q, M)>0$ and $|K|=\infty$. Then the following statements are equivalent:
\begin{itemize}
\item[{(a)}]  $M$ has maximal depth with respect to $Q$,
\item[{(b)}]  $M$ is sequentially Cohen--Macaulay with respect to $Q$,
\item[{(c)}]  $M$ is Cohen--Macaulay with respect to $Q$.
\end{itemize}
\end{Corollary}
\begin{proof}
The implication $(c)\Rightarrow (b)$ is obvious. The implication  $(b)\Rightarrow (a)$ follows from Proposition \ref{mdepth}. For $(a)\Rightarrow (c)$, we need to show $\grade(Q, M)=\cd(Q, M)$. Suppose $\grade(Q, M)<\cd(Q, M)$. As $M$ is generalized Cohen--Macaulay with respect to $Q$, we have that $H^{\grade(Q, M)}_{Q}(M)$ is finitely generated. This contradicts with Theorem \ref{gradenotfg}. Therefore, $\grade(Q, M)=\cd(Q, M)$, as desired.
\end{proof}

\begin{Example}{\em
Let $S=K[x_1,x_2, y_1, y_2, y_3, y_4]$ be the standard bigraded polynomial
ring and set  $Q=(y_1, y_2, y_3, y_4)$. We set  $R=S/(\pp_1\cap\pp_2)$ where $\pp_1=(x_1,y_1, y_2)$ and $\pp_2=(x_2, y_3, y_4)$. One has $\cd(Q, R)=2$.
   The exact sequence
  $0\rightarrow R \rightarrow S/\pp_1 \oplus S/\pp_2 \rightarrow S/{\mm}\rightarrow 0$
  yields $H^0_{Q}(R)=0$ and $H^1_{Q}(R)\iso H^0_{Q}(S/\mm)\iso S/\mm$. Here $\mm$ is the unique maximal ideal of $S$. Hence $H^1_{Q}(R)$ is finitely generated. Therefore, $R$ is generalized Cohen--Macaulay with respect to $Q$. As $\grade(Q, R)=1$ and $\mgrade(Q, R)=2$, the ring $R$ has no maximal depth with respect to $Q$. }
\end{Example}

\section{Monomial ideal, tensor product and direct sum }

In the following, we discuss the maximal depth property of monomial ideals.
\begin{Proposition}
\label{cm}
Let $I\subset S=K[x_1, \dots, x_m, y_1, \dots, y_n]$ be a monomial ideal. The following statements hold:
\begin{itemize}
		\item[(a)]  If the maximal height of an associated prime of $I$ in $Q$ is $d$, then $\mgrade(Q, S/I)=n-d$.
        \item[(b)]   If $S/I$ is Cohen--Macaulay, then $S/I$ has maximal depth with respect to $P$ and $Q$.
  \end{itemize}
\end{Proposition}
\begin{proof}
 (a): Let  $I=\bigcap_{i=1}^r\qq_i$ be an irredundant irreducible decomposition of $I$ where $\qq_i$ are $\pp_i$-primary monomial ideals, see \cite{HH}. We may write $\qq_i=\qq_i^x+\qq_i^y$ where $\qq_i^x=(x_{i_1}^{\alpha_1}, \dots, x_{i_k}^{\alpha_k})$ and $\qq_i^y=(y_{i_1}^{\beta_1}\dots, y_{i_s}^{\beta_s})$ are the monomial ideals in $K[x]$ and $K[y]$, respectively. We set $\sqrt{\qq_i}=\pp_i=\pp_i^x+\pp_i^y$ for all $i$ where $\pp_i^x=\sqrt{\qq_i^x}$ and $\pp_i^y=\sqrt{\qq_i^y}$.
The ideal $I$ has the irredundant irreducible decomposition
\[
I = (\qq_1\cap\dots \cap \qq_{a_1})\cap \dots \cap (\qq_{a_{r-1}+1} \cap \dots \cap \qq_{a_t})
\]
where
\[
\height \pp_{a_{i-1}+1}^y= \dots =\height \pp_{a_i}^y = d_i^y \quad \text{for }\quad i \in \{ 1,\dots, t\};
\]
assuming $a_0 = 0$ and $d_1^y < d_2^y < \dots < d_t^y$. Observe that
 \begin{eqnarray*}
\mgrade(Q, S/I) & = &  \min\{\cd(Q, S/\pp): \pp \in \Ass(S/I)\} \\
              & = & \min\{\dim S/(P+\pp): \pp \in \Ass(S/I)\}\\
              & = & \min\{\dim S/(P+\pp^y): \pp \in \Ass(S/I)\}\\
               & = & \min\{\dim K[y]/\pp^y: \pp \in \Ass(S/I)\}\\
               &=& n-d_t^y.
\end{eqnarray*}
Fact \ref{cd}(b) provides the second step in this sequence and the remaining steps are standard.

(b): We show that $S/I$ has maximal depth with respect to $Q$. The argument for $P$ is similar.
Since $S/I$ is Cohen--Macaulay, it follows that $d^x_t < \dots <d^x_2< d^x_1$
where
\[
\height \pp_{a_{i-1}+1}^x= \dots =\height \pp_{a_i}^x = d^x_i \quad \text{for }\quad i \in \{ 1,\dots, t\};
\]
and $d_i^x+d_{i}^y=\height \pp_i$.
Observe that
\begin{eqnarray*}
 \grade(Q, S/I)&=& \dim S/I-\cd(P, S/I)\\
               &=& m+n-(d^x_t+d^y_t)-(m-d^x_t)\\
               &=& n-d^y_t\\
               &=& \mgrade(Q, S/I).
 \end{eqnarray*}
 By Fact \ref{cd}(c) explains the first step and the forth step follows from Part(a). The remaining steps are obvious.
\end{proof}

\begin{Remark}
\label{HP}{\em
The following example shows that Proposition \ref{cm}(b) is no longer true if $I$ is not a monomial ideal. Consider the hypersurface ring $R=K[x_1, x_2, y_1, y_2]/(f)$ where $f=x_1y_1+x_2y_2$. The ring $R$ is Cohen--Macaulay of dimension $3$. By Fact \ref{cd}(c) we have $\grade(Q, R)=\dim R-\cd(P, R)=3-2=1$. As $\Ass(R)=\{(f)\}$, then $\mgrade(Q, R)=\cd(Q, R)=2$. Thus $R$ has no maximal depth with respect to $Q$.

The converse Proposition \ref{cm}(b) does not hold in general.
We set $R=S/(\pp_1\cap\pp_2)$ where $S=K[x_1,x_2, y_1, y_2]$, $\pp_1=(x_1,y_1)$, $\pp_2=(x_2, y_2)$ and $Q=(y_1, y_2)$.  One has $\grade(Q, R)=\mgrade(Q, R)=1$ and $\grade(P, R)=\mgrade(P, R)=1$. Thus $R$ has maximal depth with respect to $P$ and $Q$. The ring $R$ is not Cohen--Macaulay. In fact, $\dim R=2$ and $\depth R =1$.

Note also that the ring $R$ has maximal depth with respect to $P$ and $Q$, but not maximal depth with respect to $P+Q=\mm$.
 }
\end{Remark}
In the following, we show that the maximal depth property with respect to $Q$ is preserved under tensor product and direct sum.
We first recall the following fact.
We set $K[x]=K[x_1, \dots, x_m]$ and $K[y]=K[y_1, \dots, y_n]$.
\begin{Fact}	
\label{ass}{\em Let $K$ be an algebraically closed field. Let $L$ and $N$ be two non-zero finitely generated
graded modules over $K[x]$ and $K[y]$, respectively.  We set $M=L\tensor_KN$ and consider $M$ as an $S$-module. Then
\[
\Ass_S(M)=\{ \pp_1+\pp_2: \pp_1 \in \Ass_{K[x]}(L) \quad\text{and} \quad \pp_2\in \Ass_{K[y]}(N)\},
\]
see \cite[Corollary 2.8]{HNTT}.}
\end{Fact}

\begin{Proposition}
\label{joint}
Continue with the notation and assumptions as above. Then,  $M$ has maximal depth with respect to $Q$ if and only if  $N$ has maximal depth.
\end{Proposition}
\begin{proof}
Note that $\grade(Q, M)=\depth_{K[y]} N$.
 Let $\pp \in \Ass(M)$. By Fact \ref{ass},  there exist $\pp_1\in \Ass_{K[x]}(L)$ and $\pp_2\in \Ass_{K[y]}(N)$ such that $\pp=\pp_1+\pp_2$. Observe that
 \begin{eqnarray*}
\mgrade(Q, M) & = &  \min\{\cd(Q, S/\pp): \pp \in \Ass(M)\} \\
              & = & \min\{\dim S/(P+\pp): \pp \in \Ass(M)\}\\
              & = & \min\{\dim S/(P+\pp_2): \pp_2 \in \Ass(N)\}\\
               & = & \min\{\dim K[y]/\pp_2): \pp_2 \in \Ass(N)\}\\
               &=& \mdepth_{K[y]} N.
\end{eqnarray*}
Fact \ref{cd}(b) explains the second step in this sequence and the remaining steps are standard. Therefore, the assertion follows.
\end{proof}

\begin{Proposition}
Let $M_1, \dots, M_n$ be finitely generated bigraded $S$-modules.   Then $\Dirsum_{i=1}^n M_i$  has maximal depth with respect to $Q$ if and only if there exists $j$ such that $\grade(Q, M_j)\leq \grade(Q, M_k)$ for all $k$ and $M_j$ has maximal depth with respect to $Q$.
\end{Proposition}
\begin{proof}
Suppose  $\Dirsum_{i=1}^n M_i$ has maximal depth with respect to $Q$. Thus there exists an associated prime $\pp$ of $\Dirsum_{i=1}^n M_i$ such that $\grade \big(Q, \Dirsum_{i=1}^n M_i\big)=\cd(Q, S/\pp)$. Note that $\grade\big(Q, \Dirsum_{i=1}^n M_i\big)=\grade(Q, M_s)$ where $\grade(Q, M_s) \leq  \grade(Q, M_k)$ for all $k$. Since $\pp \in \Ass\big(\Dirsum_{i=1}^n M_i\big)=\bigcup_{i=1}^n\Ass(M_i)$, it follows that $\pp \in \Ass(M_s)$ where $\grade(Q, M_s) \leq \grade(Q, M_k)$ for all $k$. Indeed, otherwise  $\pp \in \Ass(M_l)$ for some $l$ and there exists $h$ such that  $\grade(Q, M_h)<\grade(Q, M_l)$. Hence
\begin{eqnarray*}
 \grade(Q, M_s)& \leq & \grade(Q, M_h)\\
                & < &  \grade(Q, M_l) \\
                & \leq &  \cd(Q, S/\pp)\\
                & =& \grade(Q, M_s),
\end{eqnarray*}
 a contradiction. The third inequality follows from (\ref{mgrade}).  Therefore, the conclusion follows. The other implications are obvious.
\end{proof}

\section{Hypersurface rings with maximal depth}
In the following, we classify all hypersurface rings that have maximal depth with respect to $Q$.
Let $f\in S$ be a bihomogeneous element of degree $(a, b)$
and consider the hypersurface ring $R=S/fS$. We may write
\[
f=\sum_{{| \alpha|=a}\atop {| \beta|=b}}c_{\alpha \beta }x^\alpha y^\beta  \quad\text{where}\quad c_{\alpha \beta} \in K.
\]
Note that $R$ is a Cohen--Macaulay $S$-module of dimension $m+n-1.$

\begin{Theorem}
\label{hypersurface}Let $f\in S$ be a bihomogeneous  element of
degree $(a, b)$ and  $R=S/fS$ be the hypersurface ring. Then $R$ has maximal depth with respect to $Q$ if and only if one of the following conditions holds true
 \begin{itemize}

 \item[{(a)}] $f=h_1h_2h_3$ where $\deg h_1=(\alpha_1, 0)$ with $\alpha_1 > 0$, $\deg h_2=(\alpha_2, \beta_1)$ with $\alpha_2, \beta_1 > 0$, $\deg h_3=(0, \beta_2)$ with $\beta_2>0$ and  $\alpha_1+\alpha_2=a$ and $\beta_1+\beta_2=b$;
 \item[{(b)}]$f=h_2h_3$ where $\deg h_2=(\alpha_2, \beta_1)$ with $a=\alpha_2, \beta_1 > 0$, $\deg h_3=(0, \beta_2)$ with $\beta_2> 0$ and $\beta_1+\beta_2=b$;
 \item[{(c)}] $f=h_1h_3$ where  $\deg h_1=(a, 0)$ with $a \geq 0$ and $\deg h_3=(0, b)$ with $b\geq 0$.
    \end{itemize}
\end{Theorem}
\begin{proof}
Let $f=\prod_{i=1}^r f_i$ be the unique
factorization of $f$ into bihomogeneous irreducible factors $f_i$.
We may write $h_1=\prod_{i=1}^{s-1} f_i$ where $\deg f_i=(a_i, 0)$ with $a_i\geq 0$ for $i=1, \dots, s-1$, $h_2=\prod_{i=s}^{t} f_i$ where $\deg f_i=(a_i ,b_i)$ with $a_i, b_i > 0$  for $i=s,s+1, \dots, t$ and $h_3=\prod_{i=t+1}^{r} f_i$ where $\deg f_i=(0 ,b_i)$  with $b_i\geq 0$ for $i=t+1, \dots, r$. Note that
$\sum_{i=1}^r a_i=a$
 and $\sum_{i=1}^r b_i=b$. We set $\sum_{i=1}^{s-1} a_i=\alpha_1$, $\sum_{i=s}^{t} a_i=\alpha_2$, $\sum_{i=s}^t b_i=\beta_1$ and $\sum_{i=t+1}^r b_i=\beta_2$ .  Thus $\alpha_1+\alpha_2=a$ and $\beta_1+\beta_2=b$.  We consider several cases:

Case 1: Suppose that  $\alpha_1, \alpha_2, \beta_1, \beta_2 > 0$. Then
\begin{eqnarray*}
\cd(P, R) &=& \max\{\cd(P, S/\pp): \pp \in \Ass(R)\}\\
          &=& \max \{ \dim S/(Q+(f_i)): i=1, \dots, r\}\\
          &=& \max \{ \dim S/(Q+(f_i)): i=s, \dots, r\}\\
          &=& \dim S/Q\\
          &=& m.
\end{eqnarray*}
Fact \ref{cd}(b) explains the second step. Hence by Fact \ref{cd}(c) we have
\[
\grade(Q, R)=\dim R-\cd(P, R)=m+n-1-m=n-1.
\]
On the other hand,
\begin{eqnarray*}
\mgrade(Q, R) &=& \min\{\cd(Q, S/\pp): \pp \in \Ass(R)\}\\
          &=& \min \{ \dim S/(P+(f_i)): i=1, \dots, r\}\\
          &=& \min \{ \dim S/(P+(f_i)): i=t, \dots, r\}\\
          &=& n-1.
\end{eqnarray*}
Therefore, $R$ has maximal depth with respect to $Q$.

Case 2: Suppose that $\alpha_1=0$ and  $ \beta_1, \beta_2, \alpha_2=a > 0$. Then
\begin{eqnarray*}
\cd(P, R) &=& \max\{\cd(P, S/\pp): \pp \in \Ass(R)\}\\
          &=& \max \{ \dim S/(Q+(f_i)): i=s, \dots, r\}\\
          &=& \dim S/Q\\
          &=& m.
\end{eqnarray*}
Fact \ref{cd}(c) provides
\[
\grade(Q, R)=\dim R-\cd(P, R)=m+n-1-m=n-1.
\]
On the other hand,
\begin{eqnarray*}
\mgrade(Q, R) &=& \min\{\cd(Q, S/\pp): \pp \in \Ass(R)\}\\
          &=& \min \{ \dim S/(P+(f_i)): i=s, \dots, r\}\\
          &=& \min \{ \dim S/(P+(f_i)): i=t, \dots, r\}\\
          &=& n-1.
\end{eqnarray*}
Thus, $R$ has maximal depth with respect to $Q$.

Case 3: If  $\alpha_2=\beta_1=0$ and $\alpha_1, \beta_2>0$, then  $\grade(Q, R)=\mgrade(Q, R)=n-1$, and so $R$ has maximal depth with respect to $Q$, too.

Case 4: Suppose that $\beta_2=0$ and $\alpha_1, \alpha_2, \beta_1=b > 0$. A similar arguments as above shows that $n-1=\grade(Q, R) \neq \mgrade(Q, R)=n$. Thus , $R$ has no maximal depth with respect to $Q$ in this case.

Case 5: Suppose  $\alpha_1=\beta_2=0$ and $\alpha_2=a>0, \beta_1=b>0$. Then
\[
\grade(Q, R)=\dim R-\cd(P, R)=m+n-1-m=n-1,
\]
 and $\mgrade(Q, R)=n$. Hence $R$ has no maximal depth with respect to $Q$.
Now the desired conclusion follows from the above observations.

\end{proof}

\end{document}